\documentclass[11pt]{article}

\usepackage[T1]{fontenc}
\usepackage[utf8]{inputenc}
\usepackage{lmodern}                
\usepackage{geometry}               
\usepackage{mathtools}              
\usepackage{amsthm, amssymb}
\usepackage{bm}                     
\usepackage{microtype}              
\usepackage[numbers,sort&compress]{natbib}
\usepackage{hyperref}               
\hypersetup{
    colorlinks=true,
    linkcolor=blue,
    citecolor=blue,
    urlcolor=blue,
    pdfauthor={Mikhail Ermakov},
    pdftitle={Bahadur asymptotic efficiency in the zone of moderate deviation probabilities}
}
\usepackage[capitalize]{cleveref}   

\numberwithin{equation}{section}

\theoremstyle{plain}
\newtheorem{theorem}{Theorem}[section]

\newtheorem{lemma}{Lemma}[section]
\newtheorem{remark}{Remark}[section]
\newtheorem{corollary}{Corollary}[section]

\newcommand{\thetab}{\bm{\theta}}
\newcommand{\phib}{\bm{\phi}}

\newcommand{\eb}{\bm{e}}
\newcommand{\Thetab}{\bm{\Theta}}
\newcommand{\taub}{\bm{\tau}}
\newcommand{\xb}{\bm{x}}
\newcommand{\Sb}{\bm{S}}
\newcommand{\vb}{\bm{v}}
\newcommand{\Vb}{\bm{V}}
\newcommand{\Gb}{\bm{G}}
\newcommand{\Fb}{\bm{F}}

\title{Bahadur asymptotic efficiency in the zone of moderate deviation probabilities}
\author{Mikhail Ermakov}
\date{April 2025}

\begin{document}
\maketitle

\begin{center}
Institute of Problems of Mechanical Engineering RAS, Bolshoy pr., 61, VO, 199178 St. Petersburg and\\
St. Petersburg State University, Universitetsky pr., 28, Petrodvoretz, 198504 St. Petersburg, RUSSIA
\end{center}


\footnotetext[1]{The research was supported by the Ministry of Science and Higher Education of Russian Federation (project 124041500008-1).}

\begin{abstract}
We establish an analogue of the lower bound for Bahadur asymptotic efficiency in the zone of moderate deviation probabilities. The assumptions coincide with the assumptions under which the Hajek-Le Cam locally asymptotically minimax lower bound holds. The lower bound for local Bahadur asymptotic efficiency is a special case of this lower bound.
\end{abstract}

\section{Introduction \label{s1}}
Let $X_1,\ldots,X_n$ be independent identically distributed random variables having probability measure $\mathbf{P}_{\thetab}$, $\thetab \in \Thetab$, defined on probability space $(S,\mathcal{B})$. The set $\Thetab$ is an open subset of $\mathbb{R}^d$. The value of parameter $\thetab$ is unknown. We are interested in the lower bounds of asymptotic efficiency for estimators of parameter $\thetab$.

There are two approaches for the study of lower bounds for asymptotic efficiency. One of them is a locally asymptotically minimax lower bound. In this setting we get the lower bound of the supremum of risks of statistical estimators for any small neighborhood of the true value of parameter. Such a lower bound of asymptotic efficiency does not provide precise information about the behavior of the risk of statistical estimator for a specific value of the parameter. The Hajek-Le Cam \cite{ha,le,ih,va} locally asymptotically minimax lower bound specifies a lower bound on the asymptotic efficiency for estimators $\hat\theta_n= \hat\theta_n(X_1,\ldots,X_n)$ that deviate from the true value of the parameter $\theta$ by an order of $n^{-1/2}$. Locally asymptotically minimax lower bounds for the risks in the zone of large deviation probabilities have been established by Pukhalskii and Spokoiny \cite{spok}.

The lower bound for the Bahadur asymptotic efficiency \cite{ba60,ba67,ba80,ih} provides the lower bounds for the risks of statistical estimators at each specific point of possible parameter values. In this case, only consistency of estimates is assumed. However this lower bound works only for probabilities of large deviations of statistical estimators. In confidence estimation the significance level is small. This causes the interest to the study of large and moderate deviation probabilities of estimators.

In the zone of moderate deviation probabilities, one can prove lower bounds for asymptotic efficiency both in the locally asymptotically minimax setting and in the Bahadur setting. Lower bounds for locally asymptotically minimax risks of statistical estimators have been established in \cite{er03,er12,ra}. For logarithmic asymptotics, locally asymptotically minimax lower bounds for the moderate deviation probabilities were obtained under the same assumptions \cite{er03,ra} as the lower bound for the local asymptotic efficiency of Hajek-Le Cam \cite{ha,le,ih,va}. For the strong asymptotics of moderate deviation probabilities of statistical estimators, the lower bound of the local asymptotically minimax risks has been proved under not very strong additional assumptions \cite{er03,er12}.

The goal of present paper is to obtain an analogue of the Bahadur lower bound for asymptotic efficiency of estimators in the zone of moderate deviation probabilities assuming the same conditions as in the lower bound for the Hajek-Le Cam local asymptotic efficiency. Note that the straightforward application of Bahadur method of the proof leads to significant additional conditions \cite{fu,ih} even for establishing the lower bound for the local Bahadur asymptotic efficiency. The lower bound for the local Bahadur asymptotic efficiency is a special case of the lower bound for the Bahadur asymptotic efficiency in the zone of moderate deviations. For $\Theta \subset \mathbb{R}^1$, we also note that Bahadur lower bound holds separately on each side of the exterior of the confidence interval. A multidimensional analogue of these one-sided lower bounds was also obtained.

The paper is organized as follows. All lower bounds for parametric estimation are collected in Section~\ref{o2}. In Subsection~\ref{o21} we introduce the condition under which these results are proved. This condition is the same one under which the locally asymptotically minimax Hajek-Le Cam lower bound is established. In Subsection~\ref{o22}, for the sake of completeness of presentation and better understanding of further results, a locally asymptotically minimax lower bound for moderate deviation probabilities of statistical estimates is presented. In Subsection~\ref{o23}, the analogue of the lower bound for Bahadur asymptotic efficiency of statistical estimators is provided in the zone of moderate deviation probabilities. The lower bounds for Bahadur local asymptotic efficiency are particular cases of these ones. These results are provided for the case of independent identically distributed observations. In Subsection~\ref{o231} we extend these results to the inhomogeneous case. In Subsection~\ref{o232} we show that the same results hold for the problem of parameter estimation of signal in Gaussian white noise. In Subsection~\ref{o24} we give a generalization of Bahadur lower bounds to the multivariate case, covering ``one-sided'' lower bounds, and show that its proof is no different from the proof of traditional Bahadur lower bounds on the asymptotic efficiency \cite{ba80, ih}. In Section~\ref{o3} we extend these results to the case of estimation of differentiable statistical functional value. In Section~\ref{o4} we provide the proof that the local Bahadur efficiency is a particular case of the efficiency for moderate deviation probabilities.

We use letters $c$ and $C$ as a generic notation for positive constants. For the case of multidimensional parameter $\Thetab \subset \mathbb{R}^d$, $d > 1$, we shall denote by bold letters $\thetab, \taub, \phib, \ldots$ vectors or vector functions. For $\Theta \subset \mathbb{R}^1$, for similar notation we shall implement letters $\theta, \tau, \phi, \ldots$. For any vector $\taub \in \mathbb{R}^d$ we denote by $\taub^T$ the transposed vector. We denote $\mathbf{1}(A)$ the indicator of an event $A$.
For any two sequences of positive real numbers $a_n$ and $b_n$, $a_n = o(b_n)$ implies $a_n/b_n \to 0$ as $n \to \infty$.

\section{Lower bounds for parametric estimation \label{o2}}
\subsection{Main Assumption \label{o21}}
Suppose probability measures $\mathbf{P}_{\thetab}$, $\thetab \in \Thetab$, are absolutely continuous with respect to probability measure $\mu$, defined on the same $\sigma$-algebra $\mathcal{B}$ of set $S$, and have the densities
\begin{equation*}
f(x,\thetab) = \frac{d \mathbf{P}_{\thetab}}{d\mu} (x), \quad x \in S.
\end{equation*}
For any $\thetab,\thetab_0\in \Thetab$, denote $\mathbf{P}^a_{\thetab, \thetab_0}$ and $\mathbf{P}^s_{\thetab, \thetab_0}$ absolutely continuous and singular components of probability measure $\mathbf{P}_{\thetab}$ with respect to probability measure $\mathbf{P}_{\thetab_0}$.
For all $\thetab_0,\thetab_0 + \taub \in \Thetab$ define the function
\begin{equation*}
g(x,\taub)= g(x,\thetab_0,\thetab_0 +\taub)=  \left(\frac{f(x,\thetab_0+\taub)}{f(x,\thetab_0)}\right)^{1/2} -1
\end{equation*}
for all $x$ belonging to the support of measure $\mathbf{P}^a_{\thetab_0+\taub, \thetab_0}$ and equals to zero otherwise.

We say that statistical experiment $\mathcal{E}= \{(S,\mathcal{B}), \mathbf{P}_{\thetab}, \thetab \in \Thetab\}$ has the finite Fisher information at the point $\thetab_0 \in \Thetab$, if there exists vector function $\phib\, :\, S \to \mathbb{R}^d$ such that we have
\begin{equation}\label{s21}
\int_S(g(x,\taub) - \taub^T\,\phib(x))^2\,d\mathbf{P}_{\thetab_0} = o(|\taub|^2), \quad \mathbf{P}^s_{\thetab_0,\thetab_0 + \taub}(S) = o(|\taub|^2)
\end{equation}
as $\taub \to 0$ and the matrix
\begin{equation*}
I(\thetab_0) = 4 \int_S \phib\,\phib^T \, d\mathbf{P}_{\thetab_0}
\end{equation*}
is positive definite.
Matrix $I(\thetab_0)$ is called Fisher information matrix.
We explore moderate deviation probabilities for the zone of moderate deviations defined by a sequence $u_n > 0$, $u_n \to 0$, $n u_n^2 \to \infty$ as $n \to \infty$.
We establish lower bounds for Bahadur efficiency for some fixed point $\thetab_0 \in \Thetab$.
We shall omit index $\thetab_0$ in symbols of $\mathbf{E}_{\thetab_0}$ and $\mathbf{P}_{\thetab_0}$.

\subsection{Locally asymptotically minimax lower bound \label{o22}}
Locally asymptotically minimax lower bound for risks in the zone of moderate deviation probabilities does not require any conditions of consistency.
\begin{theorem}\label{t0} 
Let $d=1$ and let the statistical experiment $\mathcal{E}= \{(S,\mathcal{B}), \mathbf{P}_\theta, \theta \in \Theta\}$ have the finite Fisher information at the point $\theta_0 \in \Theta$. Then, for any estimator $\hat\theta_n$, for the points $\theta_0,\theta_n=\theta_0 +2u_n \in \Theta$, we have
\begin{equation}\label{s2201}
\liminf_{n\to \infty}\sup_{\theta=\theta_0,\theta_n}(nu_n^2I(\theta)/2)^{-1} \log \mathbf{P}_\theta (\,|\hat\theta_n - \theta| > u_n) \ge - 1.
\end{equation}
\end{theorem}
The proof is based on the following lower bound for efficiency in hypothesis testing. Let us consider the problem of hypothesis testing $H_0: \theta = \theta_0$ versus alternative $H_n: \theta = \theta_0+v_n$, $v_n = 2 u_n$. For any test $K_n$, denote $\alpha(K_n)$ and $\beta(K_n)$ respectively type I and type II error probabilities of test $K_n$. By Theorem 2.2 in \cite{er03}, for any test $K_n$, if $\alpha(K_n) < c < 1$ and $\beta(K_n) < c < 1$, then we have
\begin{equation}\label{s270}
\limsup_{n\to \infty}(nv_n^2I(\theta_0))^{-1/2}(|2 \log \alpha(K_n)|^{1/2} + |2 \log \beta(K_n)|^{1/2}) \le 1.
\end{equation}

\begin{proof}[Proof of Theorem~\ref{t0}] 
Proof of \eqref{s270} is based on the fact that \eqref{s270} is valid for testing a similar hypothesis for the normal distribution by virtue of the Neyman-Pearson Lemma and the extension of a certain version of the statement about local asymptotic normality to the zone of moderate deviations. The lower bound \eqref{s270} has been proved simultaneously with the author in \cite{bm} with rather severe assumptions.

Define the test $K_n = K_n(X_1,\ldots,X_n) = \mathbf{1}(\hat\theta_n -\theta_0 > u_n)$.
We have
\begin{equation}\label{s271}
\alpha(K_n) \le \mathbf{P}_{\theta_0}(|\theta_n - \theta_0| > u_n)
\end{equation}
and
\begin{equation}\label{s272}
\beta(K_n) = \mathbf{P}_{\theta_n}(\hat\theta_n - \theta_0 < u_n) = \mathbf{P}_{\theta_n}(\hat\theta_n - \theta_0 - 2  u_n< - u_n)\le \mathbf{P}_{\theta_n}(|\hat\theta_n - \theta_n|> u_n).
\end{equation}
By \eqref{s270}--\eqref{s272}, we get \eqref{s2201}.
\end{proof}

\subsection{Bahadur efficiency for moderate deviations. Independent identically distributed random observations \label{o23}}
We say that estimator $\hat\thetab_n=\hat\thetab_n(X_1,\ldots,X_n)$ of parameter $\thetab \in \Thetab$ is $u_n$-consistent, if, for any $\thetab_0\in \Thetab$, there is a vicinity $U$ of $\thetab_0$, such that, for any $\delta > 0$, we have
\begin{equation*}
\lim_{n\to \infty} \sup_{\theta \in U} \mathbf{P}_\theta(\,|\hat\thetab_n - \thetab| > \delta u_n) = 0.
\end{equation*}
Let us present an analogue of lower bound for Bahadur efficiency in the zone of moderate deviation probabilities for $d=1$.
\begin{theorem}\label{t1}  
Let statistical experiment $\mathcal{E}= \{(S,\mathcal{B}), \mathbf{P}_\theta, \theta \in \Theta\}$ have the finite Fisher information at the point $\theta_0 \in \Theta\subset \mathbb{R}^1$. Let estimator $\hat \theta_n$ be $u_n$-consistent. Then, for any point $\theta_0 \in \Theta$, we have
\begin{equation}\label{s23}
\liminf_{n\to \infty}(nu_n^2I(\theta)/2)^{-1} \log \mathbf{P} (\,|\hat\theta_n - \theta_0| > u_n) \ge - 1.
\end{equation}
Moreover, we have
\begin{equation}\label{s24}
\liminf_{n\to \infty}(nu_n^2I(\theta)/2)^{-1} \log \mathbf{P} (\,\hat\theta_n - \theta_0 > u_n) \ge - 1.
\end{equation}
\end{theorem}
The multidimensional version of Theorem~\ref{t1} is provided below. In the following Theorems~\ref{t2} and \ref{t3} the asymptotic efficiency varies in different directions. This is the main distinguishing feature of theorems.
\begin{theorem}\label{t2}  
Let the statistical experiment $\mathcal{E}= \{(S,\mathcal{B}), \mathbf{P}_{\thetab}, \thetab \in \Thetab\}$ have the finite Fisher information at the point $\thetab_0 \in \Thetab\subset \mathbb{R}^d$. Let estimator $\hat \thetab_n$ be $u_n$-consistent. Then, for any point $\thetab_0 \in \Thetab$, for any open set $\Vb \subset \mathbb{R}^d$, we have
\begin{equation}\label{s25}
\liminf_{n\to \infty}(nu_n^2)^{-1} \log \mathbf{P} (\,\hat\thetab_n - \thetab_0 \in u_n  \Vb) \ge - \frac{1}{2}\inf_{\taub \in \Vb} \taub^T I(\thetab_0)\taub
\end{equation}
\end{theorem}
Note that the requirement \eqref{s21} of differentiability of the function $g$ in $\mathbb{L}_2$ can be replaced by the weaker conditions (2.5)-(2.7) in \cite{er03} on the behavior of the function $g$ itself.

The following lower bound for local asymptotic Bahadur efficiency is deduced from Theorem~\ref{t2}.
Estimator $\hat\thetab_n$ is called consistent, if for any $\thetab_0\in\Theta$, for any $\varepsilon > 0$, there is
\begin{equation*}
\lim_{n\to \infty} \mathbf{P} (\,|\hat\thetab_n - \thetab_0| > \varepsilon) = 0.
\end{equation*}
\begin{theorem}\label{t3} 
Let the statistical experiment $\mathcal{E}= \{(S,\mathcal{B}), \mathbf{P}_{\thetab}, \thetab \in \Thetab\}$ have the finite Fisher information at the point $\thetab_0 \in \Thetab\subset \mathbb{R}^d$. Let estimator $\hat \thetab_n$ be consistent. Then, for any bounded open set $\Vb \subset \mathbb{R}^d$, we have
\begin{equation}\label{d27}
\liminf_{u\to 0}\lim_{n\to \infty}(nu^2)^{-1} \log \mathbf{P} (\,\hat\thetab_n - \thetab_0 \in u  \Vb) \ge - \frac{1}{2}\inf_{\taub \in \Vb} \taub^T I(\thetab_0)\taub.
\end{equation}
\end{theorem}
Proving a lower bound for Bahadur's local asymptotic efficiency by Bahadur's method requires introducing additional regularity conditions for the family of probability measures $\mathbf{P}_{\thetab}$, $\thetab \in \Thetab$ (see Theorem 9.3, Chapter 1 in \cite{ih}, as well as the remark on Bahadur's lower bound (1.6) in \cite{ha} and the paper \cite{fu} devoted to this issue).

\begin{proof}[Proof of Theorem~\ref{t1}] 
We consider problem of hypothesis testing $H_0: \theta = \theta_0$ versus alternatives $H_n: \theta = \theta_0+v_n$, $v_n = r u_n$, $r >1$. Define the tests $K_n = K_n(X_1,\ldots,X_n) = \mathbf{1}(|\hat\theta -\theta_0| > u_n)$.
Since the estimator $\hat\theta_n$ is $u_n$-consistent, we can implement \eqref{s270} and we get
\begin{equation}\label{s28}
\begin{split}
& \limsup_{n\to \infty}(nv_n^2I(\theta))^{-1/2}(|2 \log \mathbf{P}_{\theta_0}(|\hat\theta_n -\theta_0| >u_n) |^{1/2} \\
& + |2 \log \mathbf{P}_{\theta_0+v_n}(|\hat\theta_n -\theta_0| <u_n) |^{1/2}) \le 1.
\end{split}
\end{equation}
By \eqref{s28}, we get
\begin{equation}\label{s29}
\liminf\sup_{n\to \infty}(nr^2u_n^2I(\theta))^{-1/2}(|2 \log \mathbf{P}_{\theta_0}(|\hat\theta_n -\theta_0| >u_n )|^{1/2}  \le 1.
\end{equation}
Since the choice of $r>1$ is arbitrary, we get \eqref{s23}.
Inequality \eqref{s24} is proved similarly. It suffices to implement \eqref{s270} to the test $K_n = \mathbf{1}(\hat\theta_n -\theta_0 > u_n)$.
Since the estimator $\hat\theta_n$ is $u_n$-consistent, then there is $n_0$, such that, for $n> n_0$,
\begin{equation}
\alpha(K_n) \le \mathbf{P}(|\hat\theta_n - \theta_0| > u_n) < c< 1.
\end{equation}
and
\begin{equation}\label{x28}
\begin{split}
\beta(K_n) &= \mathbf{P}_{\theta_0+v_n}(\hat\theta_n -\theta_0 < u_n)\\
&= \mathbf{P}_{\theta_0+v_n}(\hat\theta_n -\theta_0- v_n < u_n-v_n)\\
&\le \mathbf{P}_{\theta_0+v_n}(|\hat\theta_n -\theta_0- v_n| > (r-1) u_n) < c <1.
\end{split}
\end{equation}
Therefore, arguing similarly to \eqref{s28} and \eqref{s29}, we get \eqref{s24}.
\end{proof}

\begin{proof}[Proof of Theorem~\ref{t2}] 
Proof in the case of $\thetab\in\Thetab\subset\mathbb{R}^d$ is akin to $\theta\in\Theta\subset\mathbb{R}^1$ and also is reduced to the problem of estimating a parameter at two points.
Denote $cl(\Vb)$ the closure of set $\Vb$. The reasoning will be given if $0 \notin cl(\Vb)$.
Let $\taub_0 \in cl(\Vb)$ and let
\begin{equation}\label{s25u}
e \doteq \taub_0^T I(\thetab_0) \taub_0 = \inf_{\taub \in \Vb} \taub^T I(\thetab)\taub.
\end{equation}
For clarity the further reasoning will be provided with unite matrix $I(\thetab_0)$ and $e=1$.
Then, for any $\delta > 0$, there is such $\taub_1 \in \Vb$ that $|\taub_1 - \taub_0| < \delta$. There are $\lambda > 0$ such that $B(\taub_1, \lambda) \subset \Vb$ where $B(\taub_1, \lambda)$ is the ball having the center $\taub_1$ and radius $\lambda$. There is $r_0 >1$ such that, for $1 < r < r_0$, $B(r\taub_1,(r-1)/2) \subset B(\taub_1, \lambda)$.
Consider problem of testing hypothesis $H_0: \thetab = \thetab_0$ versus alternative $H_n: \thetab = \thetab_n = \thetab_0 + \vb_n$, where $\vb_n = r u_n \taub_1$.
Define the test $K_n = \mathbf{1}\{\hat\thetab_n - \thetab_0 \in u_n \Vb\}$.
Then we have
\begin{equation*}
\alpha(K_n) \le \mathbf{P} (|\hat\thetab_n - \thetab_0| > u_n) < c < 1
\end{equation*}
and
\begin{equation*}
\begin{split}
\beta(K_n) &= \mathbf{P}_{\theta_n}(\hat\thetab_n - \thetab_0 \notin u_n \Vb) \le \mathbf{P}_{\theta_n}(\hat\thetab_n - \thetab_n \notin u_n \Vb- \vb_n) \\
&\le \mathbf{P}_{\theta_n}(\hat\thetab_n - \thetab_n \notin u_n\,B(r\taub_1,(r-1)/2)) < c < 1,
\end{split}
\end{equation*}
since the set $u_n \Vb - \vb_n$ contains the ball $B(0,\lambda u_n)$. Therefore we can implement \eqref{s270}.
Further reasoning essentially coincides with the reasoning for proving Theorem~\ref{t2} and is omitted.
\end{proof}

\begin{proof}[Proof of Theorem~\ref{t3}] 
We carry out the reasoning for $\theta_0 \in \Theta \subset \mathbb{R}^1$.
Let's take the points $\theta_0$ and $\theta_u=\theta_0 + r\,u \in \Theta$, $r >1$, $u >0$. Using the consistency of the estimate $\hat\theta_n$, we obtain
\begin{equation*}
\lim_{u \to 0}\lim_{n\to \infty} \mathbf{P} (\,\hat\theta_n - \theta_0 > u ) = 0
\end{equation*}
and
\begin{equation*}
\lim_{u \to 0}\lim_{n\to \infty} \mathbf{P}_{\theta_u} (\,\hat\theta_n - \theta_u < (r-1) u) = 0.
\end{equation*}
Therefore, for any sequence $u_k > 0$, $u_k \to 0$ as $k \to \infty$, there exists a sequence $n_{0k}$, $n_{0k} u_k^2 \to \infty$ as $k \to \infty$, such that for any sequence $n_k > n_{0k}$, we have
\begin{equation*}
\lim_{k\to \infty} \mathbf{P} (\,\hat\theta_{n_k} - \theta_0 > u_k) = 0
\end{equation*}
and
\begin{equation*}
\lim_{k\to \infty} \mathbf{P}_{\theta_{u_k}} (\,\hat\theta_{n_k} - \theta_{u_k} < (r-1) u_k) = 0,
\end{equation*}
and these are precisely the key inequalities used in the proof of Theorem~\ref{t3} (see the proof of Theorems~\ref{t1}, \ref{t2}, and \ref{tb1}). We omit further reasoning.
\end{proof}

\subsection{Bahadur efficiency in moderate deviation zone. Independent inhomogeneous random observations \label{o231}}
Let $X_{n1},\ldots,X_{nn}$ be independent random variables defined on probability space $(\mathcal{X},S, \mathbf{P}_{\thetab ni})$, $1 \le i \le n$, $\thetab \in \Thetab$. The set $\Thetab$ is open set in $\mathbb{R}^d$.
Suppose probability measures $\mathbf{P}_{\theta ni}$, $1 \le i \le n$, are absolutely continuous with respect to some measure $\mu$ and have densities
\begin{equation*}
f_{ni}(x,\thetab)= \frac{d\,\mathbf{P}_{\thetab ni}}{d\mu}.
\end{equation*}
For any $\thetab, \thetab + \taub \in \Thetab$ define functions $g_{ni}(x,\thetab,\thetab+\taub) =f_{ni}^{1/2}(x,\thetab+ \taub)f_{ni}^{-1/2}(x,\thetab) - 1$ for all $x$ belonging the support $\mathbf{P}^a_{\thetab + \taub, \thetab,ni}$ and equal to zero otherwise.
We say that statistical experiments $\Xi_{ni}= \{(S,\mathcal{B}), \mathbf{P}_{\thetab ni}, \thetab \in \Thetab\}$, $1 \le i \le n$, have finite Fisher information at a point $\thetab_0 \in \Thetab$, if there are functions $\phib_{ni}\, :\, S \to \mathbb{R}^d$ such that we have
\begin{equation}\label{a9}
\int_{S} (g_{ni}(x,\thetab,\thetab+\taub) - \taub^T \phib_{ni}(x,\thetab))^2\, d\mathbf{P}_{\thetab ni} = o(|\taub|^2), \quad
\mathbf{P}^s_{\thetab + \taub, \thetab,ni}(S) = o(|\taub|^2)
\end{equation}
as $\taub \to 0$ and Fisher information matrix
\begin{equation*}
I_{ni}(\theta) = 4 \int_{S} \phib_{ni}(x,\theta) \phib_{ni}^T(x,\thetab)\, d\mathbf{P}_{\thetab ni}
\end{equation*}
is positively defined.
Denote
\begin{equation*}
\Upsilon_n(\theta)= \sum_{i=1}^n I_{ni}(\thetab).
\end{equation*}
Let $\theta_0 \in \Theta \subset \mathbb{R}^1$. We consider the problem of testing hypothesis $H_0: \theta = \theta_0 $ versus alternatives $H_n : \theta = \theta_0 + u_n$.
\begin{theorem} \label{tu1}  
Let statistical experiments $\Xi_{ni}$, $1 \le i \le n$, have finite Fisher information. Let $u_n \to 0$ and $u_n^2 \,\Upsilon_n(\theta_0) \to \infty$ as $n \to \infty$. Assume
\begin{equation}\label{a10}
\sup_{|u| \le u_n}\sum_{i=1}^n \mathbf{E}\,[ (\phi_{ni}(x,\theta_0+u) - \phi_{ni}(x,\theta_0))^2\,] = o(\Upsilon_n(\theta_0))
\end{equation}
as $n \to \infty$.
Assume that, for all $\varepsilon > 0$, condition of Lindeberg type is satisfied
\begin{equation*}
\lim_{n\to\infty} \Upsilon_n(\thetab_0)^{-1}  \sum_{i=1}^n \mathbf{E}\,\Bigl[ \phi_{ni}^2(x,\theta_0)
\mathbf{1}\Bigl(|\phi_{ni}(x,\theta_0)| > \varepsilon u_n^{-1}\Bigr)\Bigr] = 0.
\end{equation*}
Then, for any test $K_n$ such that $\alpha(K_n) < c < 1$ and $\beta(K_n) < c < 1$, we have
\begin{equation}\label{a12}
\limsup_{n\to \infty}(u_n^2\Upsilon_n(\theta_0))^{-1/2}(|2 \log \alpha(K_n)|^{1/2} + |2 \log \beta(K_n)|^{1/2}) \le 1.
\end{equation}
\end{theorem}
The lower bounds for local Bahadur, local Hodges-Lehman and local Chernoff asymptotic efficiencies \cite{er03} follow from the following Corollary~\ref{ctu1} of Theorem~\ref{tu1}.
Let $\theta_0 \in \Theta \subset \mathbb{R}^1$. We consider the problem of testing hypothesis $H_0: \theta = \theta_0 $ versus alternatives $H_n : \theta = \theta_0 + u$.
\begin{corollary} \label{ctu1}  
Let statistical experiments $\Xi_{ni}$, $1 \le i \le n$, have finite Fisher information. Let $\Upsilon_n(\theta_0) \to \infty$ as $n \to \infty$. Assume
\begin{equation}\label{cau10}
\lim_{u \to 0}\,\limsup_{n\to\infty}\,\Upsilon_n^{-1}(\theta_0)\,\sup_{|v| \le u}\,\sum_{i=1}^n \mathbf{E}\,[ (\phi_{ni}(x,\theta_0+v) - \phi_{ni}(x,\theta_0))^2\,] = 0.
\end{equation}
Assume that, for all $\varepsilon > 0$, condition of Lindeberg type is satisfied
\begin{equation*}
\lim_{u \to 0}\,\limsup_{n\to\infty}\, \Upsilon_n(\thetab_0)^{-1}\,  \sum_{i=1}^n\, \mathbf{E}\,\Bigl[ \phi_{ni}^2(x,\theta_0)
\mathbf{1}\Bigl(|\phi_{ni}(x,\theta_0)| > \varepsilon u^{-1}\Bigr)\Bigr] = 0.
\end{equation*}
Then, for any test $K_n$ such that $\alpha(K_n) < c < 1$ and $\beta(K_n) < c < 1$, we have
\begin{equation}\label{cau12}
\limsup_{u\to 0}\limsup_{n\to \infty}\,(u^2\Upsilon_n(\theta_0))^{-1/2}\,(|2 \log \alpha(K_n)|^{1/2} + |2 \log \beta(K_n)|^{1/2}) \le 1.
\end{equation}
\end{corollary}
Introduce the following assumption.

\noindent\textbf{D.} There are such constants $c, C$ and natural number $n_0$, that, for each $\thetab \in \Thetab\subset \mathbb{R}^d$ and $ n > n_0$, we have
\begin{equation*}
0 < c < \frac{\eb_1^T \Upsilon_n(\thetab) \eb_1}{\eb_2^T \Upsilon_n(\thetab) \eb_2} < C
\end{equation*}
for all unit vectors $\eb_1, \eb_2 \in \mathbb{R}^d$.

\begin{theorem} \label{tu2}   
Let $\Thetab\subset \mathbb{R}^d$ and let \textbf{D} hold. Let statistical experiments $\Xi_{ni}$, $1 \le i \le n$, have finite Fisher information. Let $\eb$ be unit vector in $\mathbb{R}^d$.
Let $u_n > 0$, $u_n \to 0$, $u_n^2\, \eb^T \Upsilon_n(\thetab_0) \eb \to \infty$ as $n \to \infty$.
Assume, there is $c>2$ such that we have
\begin{equation}\label{q21u}
\sup_{|u| < c\,u_n}\sum_{i=1}^n \mathbf{E}\,[ (\phib_{ni}^T(X_{ni},\thetab_0+u) \eb - \phib_{ni}^T(X_{ni},\thetab_0) \eb)^2\, ] = o( \eb^T\Upsilon_n(\theta_0) \eb)
\end{equation}
as $n \to \infty$.
Assume, for all $\varepsilon > 0$, Lindeberg type condition holds
\begin{equation}\label{qq21u}
\begin{split}
\lim_{n\to\infty} & (\eb^T\Upsilon_n(\theta_0) \eb)^{-1} \sum_{i=1}^n \mathbf{E}\,\Bigl[ (\phib_{ni}^T(X_{ni},\thetab_0) \eb)^2\\
& \times \mathbf{1}\Bigl(|\phib_{ni}^T(X_{ni},\thetab_0) \eb| > \varepsilon u_n^{-1}\Bigr)\Bigr] = 0.
\end{split}
\end{equation}
Then for any estimator $\hat\thetab_n$, for points $\thetab_0,\thetab_n = \theta_0 +2 u_n\eb \in \Theta$, we have
\begin{equation}\label{bu1}
\liminf_{n\to \infty}\sup_{\thetab=\thetab_0,\thetab_n}(u_n^2\eb^T\Upsilon_n(\theta_0) \eb/2)^{-1} \log \mathbf{P}_{\thetab} (\,|\hat\thetab_n - \thetab| > u_n) \ge - 1.
\end{equation}
Assume additionally that $\hat\thetab$ is $u_n$-consistent and \eqref{q21u}, \eqref{qq21u} hold for all unit vectors $\eb \in \mathbb{R}^d$.
Then, for any open set $\Vb\subset \mathbb{R}^d$, we have
\begin{equation}\label{qq25}
\liminf_{n\to \infty}(u_n^2\,\taub_n^T\Upsilon_n(\thetab) \taub_n/2)^{-1} \log \mathbf{P}_{\thetab} (\,\hat\thetab_n - \thetab \in u_n \Vb) \ge - 1,
\end{equation}
where vector $\taub_n \in cl(\Vb)$ satisfies the equation
\begin{equation}
\taub_n^T \Upsilon_n(\thetab)\taub_n = \inf_{\taub \in \Vb} \taub^T \Upsilon_n(\thetab)\taub.
\end{equation}
\end{theorem}
\begin{remark} 
Note that \eqref{a10} and \eqref{q21u} can be replaced respectively with
\begin{equation*}
\sup_{|u| < c u_n} \,\sum_{i=1}^n\, \mathbf{E}\,( g_{ni}(X_{ni},u) - u \phi_{ni}(X_{ni},\theta_0))^2 = o(|u_n|^2\, \Upsilon_n(\theta_0))
\end{equation*}
and
\begin{equation*}
\sup_{|u| < c u_n} \,\sum_{i=1}^n\, \mathbf{E}\,( g_{ni}(X_{ni},\thetab_0,\thetab_0+u \eb) - u \,\eb^T \phib_{ni}(X_{ni},\thetab_0))^2 = o(|u_n|^2\, \eb^T\Upsilon_n(\thetab_0)\,\eb)
\end{equation*}
as $n \to \infty$.
\end{remark}
From Theorem~\ref{tu2} we get the following Corollary for local Bahadur asymptotic efficiency.
\begin{corollary} \label{ctu2}   
Let $\Thetab\subset \mathbb{R}^d$ and let \textbf{D} hold. Let statistical experiments $\Xi_{ni}$, $1 \le i \le n$, have finite Fisher information. Let $\eb$ be unit vector in $\mathbb{R}^d$.
Let $u_n > 0$, $u_n \to 0$, $u_n^2\, \eb^T \Upsilon_n(\thetab_0) \eb \to \infty$ as $n \to \infty$.
Assume, there is $c>2$ such that we have
\begin{equation}\label{cq21u}
\limsup_{u\to 0}\limsup_{n\to\infty}\sup_{|v| < c\,u}( \eb^T\Upsilon_n(\theta_0) \eb)^{-1}\sum_{i=1}^n \mathbf{E}\,[ (\phib_{ni}^T(X_{ni},\thetab_0+v) \eb - \phib_{ni}^T(X_{ni},\thetab_0) \eb)^2\, ] = 0
\end{equation}
as $n \to \infty$.
Assume, for all $\varepsilon > 0$, Lindeberg type condition holds
\begin{equation}\label{cqq21u}
\begin{split}
\limsup_{u\to 0}\limsup_{n\to\infty} & (\eb^T\Upsilon_n(\theta_0) \eb)^{-1} \sum_{i=1}^n \mathbf{E}\,\Bigl[ (\phib_{ni}^T(X_{ni},\thetab_0) \eb)^2\\
& \times \mathbf{1}\Bigl(|\phib_{ni}^T(X_{ni},\thetab_0) \eb| > \varepsilon u^{-1}\Bigr)\Bigr] = 0.
\end{split}
\end{equation}
Then for any estimator $\hat\thetab_n$, for points $\thetab_0,\thetab_u = \theta_0 +2 u\eb \in \Theta$, we have
\begin{equation}\label{cbu1}
\liminf_{u\to 0}\liminf_{n\to \infty}\sup_{\thetab=\thetab_0,\thetab_u}(u^2\eb^T\Upsilon_n(\theta_0) \eb/2)^{-1} \log \mathbf{P}_{\thetab} (\,|\hat\thetab_n - \thetab| > u) \ge - 1.
\end{equation}
Assume additionally that $\hat\thetab$ is consistent and \eqref{cq21u}, \eqref{cqq21u} hold for all unit vectors $\eb \in \mathbb{R}^d$.
Then, for any open set $\Vb\subset \mathbb{R}^d$, we have
\begin{equation}\label{qq25b}
\liminf_{u\to 0}\liminf_{n\to \infty}(u^2\,\taub_n^T\Upsilon_n(\thetab) \taub_n/2)^{-1} \log \mathbf{P}_{\thetab} (\,\hat\thetab_n - \thetab \in u \Vb) \ge - 1,
\end{equation}
where vector $\taub_n \in cl(\Vb)$ satisfies the equation
\begin{equation}
\taub_n^T \Upsilon_n(\thetab)\taub_n = \inf_{\taub \in \Vb} \taub^T \Upsilon_n(\thetab)\taub.
\end{equation}
\end{corollary}
Proof of Theorem~\ref{tu2} is based on Theorem~\ref{tu1} and is akin to the proofs of Theorems~\ref{t0}--\ref{t2}. We omit the reasoning.

\begin{proof}[Proof of Theorem~\ref{tu1}]  
Reasoning basically repeats the proof of similar statements of Theorems 2.1 in \cite{er07} and Theorems 2.2 in \cite{er23}.
For $1 \le i \le n$ and $\varepsilon > 0$, define the events
\begin{equation}\label{qq26}
A_{ni} = A_{ni}(\varepsilon) = \{X_i: |u_n\phi_{ni}(X_{ni})| > \varepsilon\}, \quad D_{ni} = D_{ni}(\varepsilon) = \{X_i: |g_{ni}(X_{ni})| > \varepsilon\}. 
\end{equation}
Applying Chebyshev inequality and \eqref{qq21u}, we get
\begin{equation}\label{qq27}
\sum_{i=1}^n \mathbf{P}(A_{ni}) \le \varepsilon^2 u^2 \sum_{i=1}^n \int_S \phi_{ni}^2(x,\theta_0)
\mathbf{1}\Bigl(|\phi_{ni}(x,\theta_0) | > \varepsilon u_n^{-1}\Bigr)\,d\mu = o(u_n^2 \Upsilon_n).
\end{equation}
Applying \eqref{qq27}, we get
\begin{equation}\label{qq28}
\begin{split}
\sum_{i=1}^n \mathbf{P}(D_{ni}) &\le \sum_{i=1}^n \mathbf{P}(A_{ni}(\varepsilon/2)) \\
&+ \sum_{i=1}^n \mathbf{P}(|g_{ni}(X_{ni},u_n) - u_n\phi_{ni}(X_{ni})| > \varepsilon/2)= o(u_n^2 \Upsilon_n(\theta_0)),
\end{split}
\end{equation}
since, by virtue of Chebyshev inequality and the last estimate of the proof of Lemma 3.1, Ch.1 in \cite{ih}, we have
\begin{equation}\label{qq29}
\begin{split}
\sum_{i=1}^n & \mathbf{P}(|g_{ni}(X_{ni},u_n) - u_n\phi_{ni}(X_{ni})| > \varepsilon/2)\\
&\le 4\varepsilon^{-2}\sum_{i=1}^n \mathbf{E}(g_{ni}(X_{ni},u_n) - u_n^T\phib_{ni}(X_{ni}))^2 = o(u_n^2 \Upsilon_n(\theta_0)).
\end{split}
\end{equation}
For any event $A$, let $\bar A$ denote its complement.
Define event $U_n = (\cap_{i=1}^n \bar A_{ni}) \cap (\cap_{i=1}^n \bar D_{ni})$.
We have
\begin{equation}\label{qq30}
\begin{split}
\mathbf{P}(U_n) &= \prod_{i=1}^n (1 - \mathbf{P}(A_{ni}))\prod_{i=1}^n (1 - \mathbf{P}(D_{ni}))\\
&\ge \exp\Bigl\{-\sum_{i=1}^n (\mathbf{P}(A_{ni}) + \mathbf{P}(D_{ni}))\Bigr\} = \exp\{-o(nu_n^2)\}.
\end{split}
\end{equation}
For any events $A$ and $B$ denote $\mathbf{P}(A|B)$ the conditional probability of event $A$ given event $B$.
Denote
\begin{equation*}
\log L_n = \sum_{i=1}^n \log \Bigl(\frac{f_{ni}(X_{ni},\theta_0 +u_n)}{f_{ni}(X_{ni},\theta_0)}\Bigr).
\end{equation*}
Theorem~\ref{tu1} follows from \eqref{qq30} and Lemma~\ref{lem1} given below.
\begin{lemma} \label{lem1} 
Let assumptions of Theorem~\ref{tu1} be satisfied.
Then for any sequence $\delta_n$, $0< \delta_n <1$, $\delta_n u_n \Upsilon_n^{1/2}(\theta_0) \to \infty$ as $n \to \infty$ and any sequence $C_n$ such that $-(1-\delta_n) u_n \Upsilon_n^{1/2}(\theta_0)/2 < C_n$, we have
\begin{equation}\label{g12}
\log\mathbf{P} (u_n^{-1} \Upsilon_n^{-1/2}(\theta_0)\log L_n > C_n \,|\, U_n) = - \frac{( C_n+ u_n\Upsilon_n^{1/2}(\theta_0)/2)^2}{2}(1 + o(1))
\end{equation}
For any sequence $\delta_n$, $0< \delta_n <1$, $\delta_n u_n \Upsilon_n^{1/2}(\theta_0) \to \infty$ as $n \to \infty$ and any sequence $C_n$ such that $ C_n < (1-\delta_n)u_n \Upsilon_n^{1/2}(\theta_0)/2$, we have
\begin{equation}\label{g13}
\log\mathbf{P}_{\theta_0 +u_n} (u_n^{-1} \Upsilon_n^{-1/2}(\theta_0)\log L_n < C_n\,|\,U_n) = - \frac{( u_n \Upsilon_n^{1/2}(\theta_0)/2- C_n)^2}{2}(1 + o(1))
\end{equation}
as $n \to \infty$.
\end{lemma}
If event $U_n$ occurs, the proof of \eqref{g12} and \eqref{g13} is practically identical to the proof of Theorem 2.2 in \cite{er23}. We omit the reasoning.
\end{proof}

\subsection{Lower bound for estimating parameter of signal in Gaussian white noise \label{o232}}
In moderate deviation zone the problems of efficiency of statistical inference on the parameter of signal in Gaussian white noise have been explored in \cite{er13}. We obtained the lower bounds on the asymptotic efficiency similar to \eqref{s2201} for estimation and similar to \eqref{s270} for hypothesis testing. Thus, in \cite{er13} the lower bound for the asymptotically minimax setting of the signal parameter estimation has been established. At the same time, the analog of the lower bound \eqref{s270} for the hypothesis testing allows us to implement the technique of the proof of Theorems~\ref{t1}--\ref{t3} and to establish the lower bound of Bahadur's efficiency for signal estimation in the zone of moderate deviation probabilities.

Let we observe a realization of a random process $Y_\varepsilon(t)$, $t\in [0,1)$, $\varepsilon > 0$, defined by a stochastic differential equation
\begin{equation*}
d Y_\varepsilon(t) = S(t,\thetab) \,dt + \varepsilon\, dw(t).
\end{equation*}
Here $S(t,\thetab) \in \mathbb{L}_2(0,1)$ is a signal with unknown value of parameter $\thetab \in \Thetab \subset \mathbb{R}^d$ and $dw(t)$, $t \in [0,1)$ is Gaussian white noise. The set $\Thetab$ is open.
Suppose that, for any $\thetab_0 \in \Thetab$, the signal $S(t,\thetab)$ has the derivative in $\mathbb{L}_2(0,1)$ on $\thetab$ at the point $\thetab_0$, that is, there is such a vector function $\Sb_{\thetab}(t,\thetab_0) : [0,1) \to \mathbb{R}^d$, that we have
\begin{equation}\label{e231}
\int_0^1 (S(t,\thetab) - S(t,\thetab_0) - \Sb_\theta^T(t,\thetab_0)(\thetab - \thetab_0))^2 dt = o(|\thetab - \thetab_0|^2)
\end{equation}
as $\thetab \to \thetab_0$.
Fisher information matrix equals
\begin{equation*}
I(\thetab_0) = \int_0^1 \Sb(t,\thetab_0) \Sb^T(t,\thetab_0) \,dt
\end{equation*}
We say that a statistical experiment has finite Fisher information if Fisher information is positive definite and \eqref{e231} holds.

For signal parameter estimation all proofs of lower bounds on Bahadur asymptotic efficiency in the moderate deviation probability zone are based on the following analog of inequality \eqref{s270} for hypothesis testing established in \cite{er13}. We present it in the one-dimensional case.
Suppose we have a family $u_\varepsilon > 0$, $\varepsilon > 0$, such that $u_\varepsilon \to 0$, $\varepsilon^{-1}\,u_\varepsilon \to \infty$ as $\varepsilon \to 0$.
Let us consider the problem of testing hypothesis $H_0: \theta = \theta_0$ versus simple alternative $H_\varepsilon: \theta = \theta_0+u_\varepsilon$. For any test $K_\varepsilon$, $\varepsilon > 0$, denote $\alpha(K_\varepsilon)$ and $\beta(K_\varepsilon)$ respectively their type I and type II error probabilities. For any family of tests $K_\varepsilon$, $\varepsilon > 0$, by Theorem 2.1 in \cite{er13}, we have
\begin{equation}\label{su270}
\limsup_{\varepsilon \to 0}\, \varepsilon u_\varepsilon^{-1}I^{-1/2}(\theta_0)(|2 \log \alpha(K_\varepsilon)|^{1/2} + |2 \log \beta(K_\varepsilon)|^{1/2}) \le 1,
\end{equation}
if $\alpha(K_\varepsilon) < c < 1$ and $\beta(K_\varepsilon) < c < 1$.
Inequalities \eqref{s270} and \eqref{su270} coincide if we set $\varepsilon = n^{-1/2}$. This allows us to formulate Theorems~\ref{t1}--\ref{t3} almost identically. Let us demonstrate this on the base of Theorem~\ref{t2}.

We say that estimator $\hat\thetab_\varepsilon$ is $u_\varepsilon$-consistent, if for any $\thetab_0\in \Thetab$ there is a vicinity $U$ of $\thetab_0$, such that, for any $\delta > 0$, we have
\begin{equation*}
\lim_{\varepsilon \to 0} \sup_{\thetab \in U} \mathbf{P}_{\thetab}(\,|\hat\thetab_\varepsilon- \thetab| > \delta u_\varepsilon) = 0.
\end{equation*}
\begin{theorem}\label{tr2} 
Let statistical experiment have finite Fisher information for all $\thetab \in \Theta\subset\mathbb{R}^d$. Let estimator $\hat \thetab_\varepsilon$ be $u_\varepsilon$-consistent. Then, for any open set $\Vb \subset \mathbb{R}^d$, we have
\begin{equation}\label{sv25}
\liminf_{\varepsilon \to 0}\,\varepsilon^2 u_\varepsilon^{-2} \,\log \mathbf{P}_{\thetab} (\,\hat\thetab_\varepsilon - \thetab \in u_\varepsilon \Vb) \ge - \frac{1}{2}\inf_{\taub \in \Vb} \taub^T I(\thetab)\taub.
\end{equation}
\end{theorem}
\begin{corollary}\label{ctr2} 
Let statistical experiment have finite Fisher information for all $\thetab \in \Theta\subset\mathbb{R}^d$. Let estimator $\hat \thetab_\varepsilon$ be consistent. Then, for any open set $\Vb \subset \mathbb{R}^d$, we have
\begin{equation}\label{csv25}
\liminf_{u\to 0}\liminf_{\varepsilon \to 0}\,\varepsilon^2 u^{-2} \,\log \mathbf{P}_{\thetab} (\,\hat\thetab_\varepsilon - \thetab \in u \Vb) \ge - \frac{1}{2}\inf_{\taub \in \Vb} \taub^T I(\thetab)\taub.
\end{equation}
\end{corollary}

\subsection{Multidimensional lower bound of Bahadur asymptotic efficiency \label{o24}}
Below we provide the lower bound for Bahadur asymptotic efficiency, which is different in different directions.
Estimator $\hat\thetab_n$ is called the consistent estimator of parameter $\thetab \in \Thetab \subset\mathbb{R}^d$, if for any $\thetab \in \Thetab$, for any $\varepsilon > 0$, we have
\begin{equation*}
\lim_{n\to \infty} \mathbf{P}_{\thetab} (\,|\hat\thetab_n - \thetab| > \varepsilon) =0.
\end{equation*}
\begin{theorem}\label{tr4} 
Let $\hat\thetab_n$ be a consistent estimator of parameter $\thetab \in \Thetab$. Let $\thetab_0 \in \Thetab$. Let $\Omega \subset \mathbb{R}^d$ be open set such that $0 \notin \Omega$ and $\thetab_0 + \Omega \subset \Thetab$.
Then, for any $\tilde\thetab \in \thetab_0 + \Omega$, we have
\begin{equation}\label{b1}
\lim_{n\to \infty}\frac{1}{n} \log \mathbf{P}_{\thetab_0} (\,\hat\thetab_n - \thetab_0 \in \Omega) \ge - \int_S \log\frac{f(x,\tilde\thetab)}{f(x,\thetab_0)} \, f(x,\tilde\thetab) \, \mu(dx).
\end{equation}
\end{theorem}
\begin{proof} 
Proof is akin to \cite{ba80, ih}. Denote $- K$ the righthand side of \eqref{b1}. By Jensen inequality, the righthand side of \eqref{b1} is not positive.
We define the indicator function $\lambda_n= \lambda_n(\hat\thetab_n -\thetab_0) = 1$, if $\hat\thetab_n -\thetab_0\in \Omega$ and $\lambda_n= \lambda_n(\hat\thetab_n -\thetab_0) = 0$, if $\hat\thetab_n -\thetab_0 \notin \Omega$.
We put $r = n(K + \delta)$ with $\delta >0$.
Denote
\begin{equation*}
G_n= G_n(X_1,\ldots,X_n,\thetab_0,\tilde\thetab) = \prod_{j=1}^n \frac{f(X_j,\tilde\thetab)}{f(X_j,\thetab_0)}.
\end{equation*}
We have
\begin{equation}\label{b2}
\begin{split}
\mathbf{P}_{\thetab_0}(\hat\thetab_n -\thetab_0\in \Omega) &= \mathbf{E}_{\thetab_0} \lambda_n\\
&\ge \mathbf{E}_{\thetab_0}(\lambda_n\mathbf{1}(G_n < \exp\{r\})) \ge \exp\{-r\} \mathbf{E}_{\tilde\thetab}\Bigl\{\lambda_n \mathbf{1}(G_n < \exp\{r\})\Bigr\}\\
&\ge \exp\{-r\} (\mathbf{P}_{\tilde\thetab}(\hat\thetab_n -\thetab_0\in \Omega) - \mathbf{P}_{\tilde\thetab}(G_n > \exp\{r\})).
\end{split}
\end{equation}
Since $\hat\thetab_n$ is consistent estimator, we have
\begin{equation}\label{b3}
\lim_{n \to \infty}\,\mathbf{P}_{\tilde\thetab}(\hat\thetab_n -\thetab_0\in \Omega) =1.
\end{equation}
By the Law of Large Numbers, we have
\begin{equation}\label{b4}
\lim_{n \to \infty}\mathbf{P}_{\tilde\thetab}\Bigl(\frac{1}{n}|G_n - n K| > \delta/2\Bigr) = 0.
\end{equation}
By \eqref{b2}--\eqref{b4}, we get \eqref{b1}.
\end{proof}
Theorems~\ref{t1}--\ref{tu2} and \ref{tr4} show that the super efficiency of estimators \cite{le, ih, va} in the zones of large and moderate deviation probabilities is possible only if the estimator is in some sense inconsistent.
Thus we can say that the emergence of super effective estimators satisfying consistency conditions is inherent only to the $n^{-1/2}$--deviation zone.

\section{Lower bound for moderate deviation probabilities of differentiable statistical functionals \label{o3}}
We implement the standard reasoning \cite{bic, er08,lev,st,va} for the transition from the lower bounds for asymptotic efficiency in parametric setting to the semiparametric setting. The transition is based on the choice of the least favourable direction in the functional space of parameters and the definition of parametric families of distributions having score function coinciding with this direction. The lower bound of asymptotic efficiency of estimation for this parametric family of distributions coincides with the lower bound for asymptotic efficiency in semiparametric estimation.

Here we consider a rather simple model of semiparametric estimation. Since we prove the lower bound of asymptotic efficiency for moderate deviations almost for the same setting as the setting in Hajek-Le Cam lower bound \cite{ha,le,ih,va} of asymptotic efficiency, then the unique problem is the definition of $u_n$-consistency. It is clear that this condition should be satisfied uniformly for the parameters of the least favourable family of distributions.
Let $(S,\mathcal{B})$ be a measurable space, let $\Lambda$ be the set of all probability measures on $(S,\mathcal{B})$. Let $X_{1},\ldots,X_{n}$ be i.i.d.r.v.'s with probability measure $\mathbf{P} \in \Lambda$. Suppose a priori information is given that $\mathbf{P} \in \Gamma \subseteq \Lambda $. Let the functional $T: \Lambda\to \mathbb{R}^{1}$ be defined.
We want to estimate the value of the functional $T(\mathbf{P})$ when it is known that $\mathbf{P} \in \Gamma $.
We implement the standard terminology (see \cite{bic,lev,va}).
For fixed $\mathbf{P} \in \Gamma$ denote $\Pi (\Gamma ,\mathbf{P})$ the set of all maps $\lambda :u\to \mathbf{P}_{u}$ from interval $(0,\delta )$ in $\Gamma $ satisfying \eqref{s21} for some function $\phi(x)= \phi_\lambda(x)$, $\phi\in \mathbb{L}_{2}(\mathbf{P})$.
Let $\Delta (\Gamma ,P)$ be the set of all functions $\phi_\lambda$, $\lambda \in \Pi (\Gamma ,\mathbf{P})$, satisfying \eqref{s21} and let $cl(\Delta(\Gamma ,\mathbf{P}))$ be the closure in $\mathbb{L}_{2}(\mathbf{P})$ of the set $\Delta (\Gamma ,\mathbf{P})$.
Define the linear space $\mathbb{L}(\Gamma ,\mathbf{P})$ as the closure in $\mathbb{L}_{2}(\mathbf{P})$ of the linear space induced by the functions $\varphi \in \Delta (\Gamma ,\mathbf{P})$. The linear space $\mathbb{L}(\Gamma ,\mathbf{P})$ can be interpreted as a tangent space of $\Gamma $ at a point $\mathbf{P}$ for the Hellinger metric.
We say that the function $\psi _{P}\in \mathbb{L}(\Gamma ,\mathbf{P})$, $\mathbf{E}_{\mathbf{P}} \psi_\mathbf{P}(X_1) = 0$, is the influence function of the functional $T$ in $\Gamma $ at a point $\mathbf{P} \in \Gamma $, if, for all $\lambda \in \Pi (\Gamma ,\mathbf{P})$, there holds
\begin{equation*}
T(\mathbf{P}_{u}) - T(\mathbf{P}) = u \mathbf{E} [\psi_{\mathbf{P}}(X_1) \phi_\lambda(X_1)] + o(u), \quad u\downarrow 0,
\end{equation*}
and $\mathbf{E}[\psi_{\mathbf{P}}(X_1)\, \phi(X_1)] =0$ for any function $\phi \in \mathbb{L}_2(\mathbf{P})$ orthogonal to all functions $\phi_\lambda \in \Delta(\Gamma ,\mathbf{P})$.
Make the following assumption.

\noindent\textbf{E.} For all $\mathbf{P} \in \Gamma$ there exists influence function $\psi_{\mathbf{P}} $ of functional $T$ in $\Gamma $ and $\psi_{\mathbf{P}}\in cl(\Delta(\Gamma ,\mathbf{P}))$.

Denote $I(\Gamma,\mathbf{P}) = (\mathbf{E}_{\mathbf{P}} \psi_{\mathbf{P}}^2(X_1))^{-1}$.
We say that estimator $\hat\theta_n=\hat\theta_n(X_1,\ldots,X_n)$ of a value of functional $T(\mathbf{P})$ is $u_n$-consistent, if, for any $\mathbf{P}_0\in \Gamma$, for any $\delta > 0$, we have
\begin{equation}\label{ququ10}
\lim_{n\to \infty} \sup_{\mathbf{P} \in U_\varepsilon} \mathbf{P}(\,|\hat\theta_n - T(\mathbf{P})| > \delta u_n) = 0,
\end{equation}
for any vicinity $U=U(T,\mathbf{P}_0) = \{\mathbf{P}: |T(\mathbf{P}) - T(\mathbf{P}_0)| < Cu_n, \mathbf{P} \in \Gamma\}$, $C> 3\delta$, of the point $\mathbf{P}_0\in \Gamma$.
If $\psi_{\mathbf{P}_0}\in\Delta (\Gamma ,\mathbf{P}_0)$, then we can replace $ U$ in \eqref{ququ10} with image $ \lambda_{\psi_{\mathbf{P}_0}} ((0,Cu_n))$.

\begin{theorem}
Assume \textbf{E}. Let $u_{n}\to 0$, $nu^{2}_{n}\to\infty $ as $n\to\infty $.
Then, for any $u_n$-consistent sequence of estimators $\hat\theta_{n}$, we have
\begin{equation*}
\liminf_{n\to\infty } (nu^{2}_{n}\,I(\Gamma,\mathbf{P})/2)^{-1}\log \mathbf{P}(|\hat\theta_{n}- T(\mathbf{P})|> u_{n}) \ge -1.
\end{equation*}
Moreover, we have
\begin{equation*}
\liminf_{n\to\infty } (nu^{2}_{n}\,I(\Gamma,\mathbf{P})/2)^{-1}\log \mathbf{P}(\hat\theta_{n}- T(\mathbf{P}) > u_{n}) \ge -1.
\end{equation*}
\end{theorem}

We call estimator $\hat\theta_n$ consistent, if, for any $\mathbf{P}_0\in \Gamma$, for any $\delta > 0$, we have
\begin{equation}\label{ququ15}
\lim_{n\to \infty} \mathbf{P}_0(\,|\hat\theta_n - T(\mathbf{P}_0)| > \delta) = 0.
\end{equation}
\begin{theorem}\label{t10}  
Assume \textbf{E}. Then, for any consistent estimator $\hat\theta_n$, we have
\begin{equation*}
\liminf_{u\to 0}\lim_{n\to \infty}(n\,u^2\,I(\Gamma,\mathbf{P}_0)/2)^{-1} \log \mathbf{P}_0 (\,|\hat\theta_n - T(\mathbf{P}_0)| > u ) \ge - 1.
\end{equation*}
Moreover, we have
\begin{equation*}
\liminf_{u\to 0}\lim_{n\to \infty}(n\,u^2\,I(\Gamma,\mathbf{P}_0)/2)^{-1} \log \mathbf{P}_0 (\,\hat\theta_n - T(\mathbf{P}_0) > u ) \ge - 1.
\end{equation*}
\end{theorem}

\section{Asymptotic efficiency for moderate deviations and local Bahadur asymptotic efficiency \label{o4}}
Bahadur asymptotic efficiency \cite{ba60,ba67,ba80} is well-known measure of quality of test and estimators. Exploration of Bahadur efficiency is a rather hard technical problem \cite{nik} and, often, is replaced with the study of local Bahadur asymptotic efficiency (see \cite{dni, er03, ih} and many other papers). In \cite{er03} we pay attention that the statements on local Bahadur asymptotic efficiency follows from similar Theorems on moderate deviation probabilities. Below we present this statement in explicit form.

In \cite{er98, er07}, we show that technique of Fréchet differentiability of statistical functionals can be implemented to the study of moderate deviations probabilities to the same extent as in the case of proving asymptotic normality \cite{se}. In \cite{gao} this result has been extended to the case of Hadamard differentiability. Thus the results \cite{er98, er07, gao} are valid for the setting of local Bahadur efficiency as well.

Let we need to estimate parameter $\thetab \in \Thetab \subset \mathbb{R}^d$. We say that estimator $\hat\thetab_n = \hat\thetab_n(X_1,\ldots,X_n)$ satisfies Moderate Deviation Principle with rate function $I \,:\, \mathbb{R}^d \to \mathbb{R}_+$, if there hold
\begin{itemize}
\item[\it i.] for any $L > 0$ the set $\{\xb: I(\xb) < L\}$ is compact,
\item[\it ii.] for any sequence $u_n > 0$, $u_n \to 0$, $nu_n^2 \to \infty$, for any open set $\Gb \subset \mathbb{R}^d$, we have
\begin{equation*}
\liminf_{n \to \infty} (nu_n^2/2)^{-1}\log \mathbf{P}_{\thetab*}(\hat\thetab_n - \thetab \in u_n \Gb) \ge -\inf_{\xb \in \Gb} I(\xb),
\end{equation*}
\item[\it iii.] for any sequence $u_n > 0$, $u_n \to 0$, $nu_n^2 \to \infty$, for any closed set $\Fb \subset \mathbb{R}^d$, we have
\begin{equation*}
\limsup_{n \to \infty} (nu_n^2/2)^{-1}\log \mathbf{P}_{\thetab}^*(\hat\thetab_n - \thetab \in u_n \Fb) \le -\inf_{\xb \in \Fb} I(\xb).
\end{equation*}
\end{itemize}
Here $\mathbf{P}_{\thetab*}(A)$ and $\mathbf{P}_{\thetab}^*(A)$ denote outer and inner probabilities of a set $A \subset \mathbb{R}^d$.

We say that estimator $\hat\thetab_n = \hat\thetab_n(X_1,\ldots,X_n)$ satisfies local Bahadur Large Deviation Principle with rate function $I \,:\, \mathbb{R}^d \to \mathbb{R}_+$, if there hold
\begin{itemize}
\item[\it i.] for any $L > 0$ the set $\{\xb: I(\xb) < L\}$ is compact,
\item[\it ii.] for any open set $\Gb \subset \mathbb{R}^d$, we have
\begin{equation*}
\lim_{u\to 0}\liminf_{n \to \infty} (nu^2/2)^{-1}\log \mathbf{P}_{\thetab*}(\hat\thetab_n - \thetab \in u \Gb) \ge -\inf_{\xb \in \Gb} I(\xb),
\end{equation*}
\item[\it iii.] for any closed set $\Fb \subset \mathbb{R}^d$, we have
\begin{equation*}
\lim_{u\to 0}\limsup_{n \to \infty} (nu^2/2)^{-1}\log \mathbf{P}_{\thetab}^*(\hat\thetab_n - \thetab \in u \Fb) \le -\inf_{\xb \in \Fb} I(\xb).
\end{equation*}
\end{itemize}

\begin{theorem} \label{tb1} 
Let estimator $\hat\theta_n$ satisfy Moderate Deviation Principle with rate function $I$. Then $\hat\theta_n$ satisfies local Bahadur Large Deviation Principle with the rate function $I$ as well.
\end{theorem}
\begin{proof}[Proof of Theorem~\ref{tb1}] 
Let we take arbitrary sequence $v_k \to 0$ as $k \to \infty$. For proof of local Bahadur Large Deviation Principle, it suffices to show that, for any subsequence $n_k \to \infty$ and $n_k v_k^2 \to \infty$ as $k \to \infty$, there hold
\begin{equation*}
\liminf_{k \to \infty} (n_k v_k^2/2)^{-1}\log \mathbf{P}_{\thetab*}(\hat\thetab_{n_k} - \thetab \in v_k G) \ge -\inf_{\xb \in \Gb} I(\xb)
\end{equation*}
and
\begin{equation*}
\limsup_{k \to \infty} (n_k v_k^2/2)^{-1}\log \mathbf{P}_{\thetab}^*(\hat\thetab_{n_k} - \thetab \in v_k F) \le -\inf_{\xb \in \Fb} I(\xb).
\end{equation*}
However this is Moderate Deviation Principle, if we put $u_{n_k}= v_k$.
\end{proof}

\end{document}